\documentclass[12pt,a4paper]{amsart}
\usepackage{mathtools,amssymb,amsmath,amsopn,amsthm,graphicx,MnSymbol,centernot,stmaryrd}
\usepackage{tikz}
\usepackage{enumitem}
\usepackage{amsaddr}
\usepackage{etoolbox}
\patchcmd{\subsection}{-.5em}{.5em}{}{}
\usetikzlibrary{matrix}
\begin{document}

\newtheorem{definition}{Definition}[section]
\newtheorem{definitions}[definition]{Definitions}
\newtheorem{deflem}[definition]{Definition and Lemma}
\newtheorem{lemma}[definition]{Lemma}
\newtheorem{proposition}[definition]{Proposition}
\newtheorem{theorem}[definition]{Theorem}
\newtheorem{corollary}[definition]{Corollary}
\newtheorem{cors}[definition]{Corollaries}
\theoremstyle{remark}
\newtheorem{remark}[definition]{Remark}
\theoremstyle{remark}
\newtheorem{remarks}[definition]{Remarks}
\theoremstyle{remark}
\newtheorem{notation}[definition]{Notation}
\theoremstyle{remark}
\newtheorem{example}[definition]{Example}
\theoremstyle{remark}
\newtheorem{examples}[definition]{Examples}
\theoremstyle{remark}
\newtheorem{dgram}[definition]{Diagram}
\theoremstyle{remark}
\newtheorem{fact}[definition]{Fact}
\theoremstyle{remark}
\newtheorem{illust}[definition]{Illustration}
\theoremstyle{remark}
\newtheorem{rmk}[definition]{Remark}
\theoremstyle{definition}
\newtheorem{que}[definition]{Question}
\theoremstyle{definition}
\newtheorem{conj}[definition]{Conjecture}
\newtheorem{por}[definition]{Porism}

\renewenvironment{proof}{\noindent {\bf{Proof.}}}{\hspace*{3mm}{$\Box$}{\vspace{9pt}}}
\author[Shukla, Jain, Kuber]{Himanshu Shukla, Arihant Jain and Amit Kuber}
\address{Department of Mathematics and Statistics\\Indian Institute of Technology, Kanpur\\Uttar Pradesh, India}
%\affil{}
\email{hshukla@iitk.ac.in, arihantj@iitk.ac.in, askuber@iitk.ac.in}
\title{{Definable combinatorics with dense linear orders}}
\keywords{Grothendieck ring, pigeonhole principle, dense linear order}
\subjclass[2010]{03C64, 06A05, 18F30}

\begin{abstract}
We compute the model-theoretic Grothendieck ring, $K_0(\mathcal{Q})$, of a dense linear order (DLO) with or without end points, $\mathcal{Q}=(Q,<)$, as a structure of the signature $\{<\}$, and show that it is a quotient of the polynomial ring over $\mathbb{Z}$ generated by $\mathbb N_+\times(Q\sqcup\{-\infty\})$ by an ideal that encodes multiplicative relations of pairs of generators. As a corollary we obtain that a DLO satisfies the pigeon hole principle (PHP) for definable subsets and definable bijections between them--a property that is too strong for many structures.
\end{abstract}

\maketitle

\newcommand\A{\mathcal{A}}
\newcommand\C{\mathcal{C}}
\newcommand\D{\mathcal{D}}
\newcommand\LL{\mathcal{L}}
\newcommand\M{\mathcal{M}}
\newcommand\T{\mathcal{T}}

\newcommand{\N}{\mathbb{N}} 
\newcommand{\R}{\mathbb{R}}
\newcommand{\Z}{\mathbb{Z}}
\newcommand{\Dphi}{\mathcal{D}^{\emptyset}_n}
\newcommand{\Dabar}{\mathcal{D}^{\overline{a}}_n}
\newcommand{\Dbbar}{\mathcal{D}^{\overline{b}}_n}
\newcommand{\Dnq}{\mathcal{D}^{Q}_n}

\newcommand{\Hb}{\mathbb{H}}
\newcommand{\Q}{\mathcal{Q}}
\newcommand{\la}{\langle}
\newcommand{\ra}{\rangle}
\newcommand{\Def}{\mathrm{Def}}
\newcommand{\Deft}{\widetilde{\mathrm{Def}}}
\newcommand{\Atphi}{At_n^{\emptyset}}
\newcommand{\At}{At_n^{\overline{a}}}
\newcommand{\Col}{\mathrm{Color}(n,\overline{a})}
\newcommand{\Split}{\mathrm{Split}^{\overline{a}}_n}
\newcommand{\LC}{\chi_{T}}
\newcommand{\GC}{\chi_n^{\overline{a}}}
\renewcommand{\qedsymbol}{$\blacksquare$}
\newcommand{\Rplus}{[0, \infty)}
\newcommand{\Slz}{SL\_2(\Z)}
\newcommand{\PP}{\mathcal{P}}
\newcommand{\res}{\upharpoonright}
\newcommand{\MLR}{Martin L\"of random}
\newcommand{\lep}{\stackrel{\scriptstyle +}{\le}}
\newcommand{\mmid}{\quad\mid\quad}
\newcommand*\quot[2]{{^{\textstyle #1}\big/_{\textstyle #2}}}

\section{Introduction}
What of elementary combinatorics holds true in a class of first order structures if sets, relations, and maps must be definable? From a purely model-theoretic point of view Kraji\v{c}ek and Scanlon \cite{scanlon} studied definable versions of some combinatorial principles, notably the pigeonhole principle (PHP), in various structures. Such questions about definable sets can be reformulated in terms of certain properties of an algebraic invariant--the Grothendieck ring--associated with the structure. The Grothendieck ring of a structure $\A=(A,\hdots)$, that classifies sets definable with parameters from $A$ up to definable bijections with disjoint union as addition and cartesian product as multiplication, was motivated by the use of the Grothendieck ring of varieties in the theory of motivic integration.

In this paper we study the Grothendieck ring of a dense linear order (DLO) $\Q=(Q,<)$ thought of as a structure of signature consisting of a single binary relation symbol $<$. For the most part we focus our attention to the case when $\Q$ does not have any end points. The theory of DLOs without end points is complete, $\aleph_0$-categorical, and admits complete elimination of quantifiers. Quantifier elimination plays a crucial role in the analysis of definable sets; we get a ``generating set'' for the boolean algebra of definable subsets of a fixed power of $\Q$. We reduce the union of such generators for all powers using product relations to a much smaller collection, the equivalence classes of which are indexed by $\N_+\times(\mathbb Q\sqcup\{-\infty\}$. The Grothendieck ring, $K_0(\Q)$, is the polynomial ring over integers with the above mentioned elements as generators modulo an ideal generated by some combinatorial multiplicative relations (Theorem \ref{iso}).

Further using the analysis of definable subsets of DLOs without end points we show that the Grothendieck ring of an infinite DLO with at least one end point is isomorphic to the Grothendieck ring of the DLO obtained by removing all the end points (Theorem \ref{wep}). As a consequence we show that DLOs satisfy PHP (Corollary \ref{PHP}). In passing we also observe that DLOs do not satisfy both counting of cardinalities principles (CC1 and CC2) (Remark \ref{CCrmk}). We conclude the study of DLOs by proving some interesting combinatorial properties of $K_0(\Q)$.

In his PhD thesis \cite{kuber}, the third author studied Grothendieck rings of two other classes of structures admitting some form of elimination of quantifiers, namely modules and algebraically closed fields. As a result he stated a \emph{transfer principle} \cite[Question~8.2.3]{kuber} which asks whether relations between the generators of the semiring of definable bijection classes of definable sets, if such generators exist, are transferred to the generators of the K-groups, which in our case refers only to $K_0$. Cluckers and Halupczok \cite[\S 1]{cluc_hal} point out that the existence of generators for a semiring is a highly non-trivial. In this paper we demonstrate that the transfer principle holds even for DLOs by exhibiting a simple enough generating set for the Grothendieck semiring which in turn gives a description of the Grothendieck ring in terms of generators and relations. See \S \ref{s7} for a survey of Grothendieck rings of some structures.

The paper is organized as follows. In \S \ref{s7} we briefly recall the construction of the Grothendieck ring, state versions of PHP and describe their interrelations. For a DLO $\Q$ without end points, the structure of atoms of the finite boolean algebra of subsets of $\Q^n$ definable with a finite parameter set are identified in \S \ref{s8}. 
Functions called `global characteristics' that count the number of ``similarity types'' of atoms in the canonical atomic decomposition of a definable set are used in \S \ref{s9} to characterize definable sets up to definable bijections (Theorem \ref{th:semiring}). The semiring thus formed is cancellative (Theorem \ref{cancellative}) and thus embeds into $K_0(\Q)$, 
the computation of which takes up entire \S \ref{s10}. In \S \ref{s12} we compute the Grothendieck ring of an infinite DLO with at least one end point. Finally some interesting combinatorial properties of $K_0(\Q)$ are stated in \S \ref{s11}.

\section{Grothendieck ring of a structure}\label{s7}

Let $\LL$ denote a signature and $\A=(A,\cdots)$ denote a first order $\LL$-structure. We briefly recall the construction of the Grothendieck ring, denoted $K_0(\A)$, of the structure $\A$ from \cite{scanlon}. A definable subset of $\A$ always means a set definable by an $\LL$-formula with parameters from $\A$. A definable bijection between definable sets $D_1\subseteq\A^n$ and $D_2\subseteq\A^m$ is a bijection $f:D_1\to D_2$ whose graph is a definable subset of $\A^{n+m}$. We use the notation $\overline{\Def}(\A)$ to denote the collection of definable subsets of $\A^n$ for arbitrary $n$, and the notation $\Deft(\A)$ to denote the collection of equivalence classes of elements of $\overline{\Def}(\A)$ under definable bijections. We denote the surjective map taking a definable set to its definable bijection equivalence class by $[]:\overline{\Def}(\A)\to\Deft(\A)$. The codomain can be equipped with a semiring structure:
\begin{itemize}
\item $0:=[\{\emptyset\}]$;
\item $1:=[\{a\}]$ for any $a\in A$;
\item $[A]+[B]:= [A'\sqcup B']$, where $A'\cap B'=\emptyset$, $A'\in[A]$, $B'\in[B]$;
\item $[A]\cdot[B]:=[A\times B]$.
\end{itemize}
We say that a semiring is cancellative if $\alpha+\gamma=\beta+\gamma\ \Rightarrow \alpha=\beta$.

The Grothendieck ring $K_0(\A)$ is the quotient of $\Deft(\A)\times\Deft(\A)$ by the equivalence relation defined by $(\alpha,\beta)\sim(\alpha',\beta')$ iff there exists $\gamma$ such that $\alpha+\beta'+\gamma=\alpha'+\beta+\gamma$. The ring thus constructed has the universal property that any semiring map from $\Deft(\A)$ to a ring factors uniquely through $K_0(\A)$.

Following \cite{scanlon} we first state the model-theoretic version of the \emph{Pigeon Hole Principle (PHP)} for a structure. 
\begin{definition}
A structure $\A$ is said to satisfy \emph{PHP} if for every $D_1\subsetneq D_2\in\overline{\Def}(\A)$ there does not exist a definable bijection between $D_1$ and $D_2$.
\end{definition}
It is possible to characterize this combinatorial principle in terms of existence of a special subset of the Grothendieck ring of the structure.
\begin{definition}[Partially ordered ring]
A commutative ring $R$ with unity is said to be \emph{partially ordered} if there exists $P\subseteq R$ such that the following conditions hold.
\begin{enumerate}
    \item $0\in P$,
    \item $1\in P$,
    \item $P+P\subseteq P$,
    \item $P\cdotp P\subseteq P$
    \item $\forall x\ne0\ (x\in P\ \Rightarrow\ -x\notin P)$.
\end{enumerate}
\end{definition}
The subset $P$ in the above definition is usually referred to as the positive part of the ring $R$. Now we are ready to state the promised characterization of PHP.

\begin{theorem}\label{ts8.2}\cite[Theorem~4.3]{scanlon}
A structure $\A$ satisfies PHP iff $(K_0(\A),P)$ is a partially ordered ring such that $[D]\in P$ for each $D\in\overline{\Def}(\A)$.
\end{theorem}
The principle PHP is very strong and is not satisfied by many structures. Kraji\v{c}ek \cite{krajicek} gave weaker version of pigeonhole principle called ontoPHP. 
\begin{definition}\label{ontophp}
A structure $\A$ is said to satisfy ontoPHP if for every definable set $D$ and a point $a\in D$, there does not exist a definable bijection between $D$ and $D\setminus \{a\}$. 
\end{definition}
There is an equivalent characterization of ontoPHP in terms of $K_0(\A)$.

\begin{proposition}\label{ontophp-th}\cite[Theorem~3.1]{krajicek}
A structure $\A$ satisfies ontoPHP iff $K_0(\A)\neq0$.
\end{proposition}
Now we state the counting of cardinalities principles.
\begin{definition}\label{CC}\cite[\S 4]{scanlon}
\begin{enumerate}
\item A structure $\A$ is said to satisfy the property \emph{CC1} if, given two definable sets $A$ and $B$, either there exists a definable injection of $A$ into $B$ or of $B$ into $A$.
\item A structure $\A$ is said to satisfy the property \emph{CC2} if, given two definable sets $A$ and $B$, either there exists a definable injection of $A$ into $B$ or a definable surjection of $B$ onto $A$.
\end{enumerate}
\end{definition}
Here we give a brief survey of structures whose Grothendieck (semi)rings have been computed. If $\A$ is a finite structure then $K_0(\A)\simeq\mathbb Z$. A theorem of Ax \cite{Ax} shows that $\mathbb C$ satisfies PHP but the exact structure of the Grothendieck ring $K_0(\mathbb{C})$ is not known. Third author showed \cite{kuber2} that a module $M_R$ over a ring $R$ satisfies ontoPHP and explicitly computed $K_0(M_R)$ as a quotient of a monoid ring exemplifying the `transfer principle' \cite[Question~8.2.3]{kuber}. Kraji\v{c}ek and Scanlon showed that the Grothendieck ring of a real closed field is isomorphic to the ring of integers. Cluckers and Haskell~\cite{cluckers},~\cite{cluc_hask} proved that both the field $\mathbb{Q}_p$ of $p$-adics, and the field $F_q((t))$ of formal Laurent series do not satisfy ontoPHP. Cluckers and Halupczok \cite{cluc_hal} computed the Grothendieck semiring of Presburger groups but showed that they do not satisfy ontoPHP.

Given a language $\LL$, the Greek capital letters $\Phi,\Psi,\Gamma,\hdots$ will denote $\LL$-formulas. Small roman letters $a,b,c,d, q$ will denote parameters. Roman letters $A,B,D,R$ will denote definable subsets of a structure, whereas letters $X, Y, Z$ will denote variables. Following the notation of \cite{hodges}, given an $\LL$-structure $\M$ and an $\LL$-formula $\Phi$ with $n$ variables and parameter set $\overline{a}$, by $\Phi(\M^n,\overline{a})$ we denote the subset of $M^n$ definable by $\Phi$. The set $\N$ of natural numbers always contains $0$.

\section{Definable subsets of a DLO without end points}\label{s8}

We work with a fixed DLO $\Q=(Q,<)$ without end points. We denote the theory of DLOs without end points by $ \T$. First we fix some notations.
\begin{itemize}
\item We denote the variable set $\{X_{n+1},X_{n+2},\ldots,X_{n+m}\}$ by $ \overline{X}[n+1:n+m]$ for $n,m\in \N$.\\[-3mm]
\item With $\overline{X}$ we refer to $\{X_{1},X_{2},\ldots,X_{n}\}$ for some $n\in\N$ that is clear from the context. We will use $\overline{X}[1:n]$ instead of $\overline{X}$ if we want to emphasize on $n$.\\[-3mm]
\item Given a quantifier-free formula $\Phi(\overline{X})$ written in the disjunctive normal form (DNF) as $\Phi(\overline X):=\bigvee\limits_{\Gamma\in\C_\Phi}\Gamma$, where $\C_{\Phi}$ is a finite set of conjunctive clauses.\\[-3mm]
\item For ease of notation $(\alpha_1<\alpha_2) \wedge (\alpha_2<\alpha_3)$ will be written as $\alpha_1<\alpha_2<\alpha_3$, where, for each $1\leq i\leq 3,\ \alpha_i$  can be a variable or a parameter.
\item Following conventions of model theory the notation $\overline{a}$ will denote an $m$-tuple of elements of $Q$ as well as the set $\{a_1,a_2,\ldots,a_m\}$; the use will be clear from the context.
\item Given parameter sets $\overline{a}$ and $\overline{b}$ in $Q$, $\overline{a}\cup\overline{b}$ will be denoted by  $\overline{a}\overline{b}$. 

\end{itemize}

\subsection{Definable sets in dimension $n$}
Fix some $n\in\N$. Let $\Dphi$ denote the set of definable subsets of $\Q^n$ that are definable by formulas without parameters. Clearly $\Dphi$ forms a finite boolean algebra under usual set operations. Given $D\in\Dphi$ such that $D=\Phi(\Q^n)$, it is possible to convert $\Phi(\overline{X})$ into a DNF with conjunctive clauses containing only positive atomic formulas for our structure is a linear order. Specifically, if $\alpha,\beta$ are variables (or parameters), then $\neg(\alpha<\beta)$ is equivalent to $(\beta<\alpha)\vee(\alpha=\beta)$ modulo theory $\T$. We now identify a subclass of $\Dphi$, the elements of which will be called `related sets', and show that they are the atoms of the boolean algebra $\Dphi$.

\begin{definition}\label{rset}
Call a definable set $D\in\Dphi$ to be a \emph{related set} if there exists a formula $\Phi(\overline{X})$ and a total ordering $\prec$ on $\{1,2,\hdots,n\}$ such that for each consecutive pair of indices $i\prec j$ precisely one of $X_i< X_j$ or $X_i=X_j$ holds.
\end{definition}
The related sets will correspond to the formulas of the form
\begin{equation}\label{e1}
    \overline{X_1}< \overline{X_2}< \ldots< \overline{X_{k-1}} < \overline{X_k}
\end{equation}
for some $k\in\N$, where $\overline{X}_{i}$ refers to a tuple, $(X_{i_1},X_{i_2},\cdots,X_{i_{l_i}})$ satisfying
$$X_{i_1}=X_{i_2}=\cdots=X_{i_{l_i}},$$
where $X_{i_t}\in \overline{X}$ for each $1\leq t\leq l_i$ and $\sum_{i=1}^kl_i=n$. For a related set $R$ let $\Phi_R(\overline{X})$ \emph{denote the formula associated to} $R$ which has a form as in Equation \eqref{e1}.
\begin{definition}\label{d2}
For a related set $R\in\Dphi$ define its \emph{height} to be the number $k$ in Equation \eqref{e1}, which is essentially the number of $\overline{X_i}$ in $\Phi_R(\overline{X})$. Denote the height of $R$ by $H(R)$.
\end{definition}
\begin{proposition}\label{p1}
The related sets in $\Dphi$ form the atoms for the boolean algebra $\Dphi$.
\end{proposition}
\begin{proof}
We show that any two distinct related sets are mutually disjoint. Let $R_1\neq R_2\in\Dphi$. Then there exists a variable pair $(X_i,X_j)$ which differs in its relation in $R_1$ and $R_2$. WLOG assume that $X_i<X_j$ in $R_1$ and $X_j<X_i$ in $R_2$. (The proof is similar if $X_i=X_j$ holds for $R_2$.) If $\overline{a}\in R_1\cap R_2$, we cannot have $a_i< a_j$ and $a_j< a_i$. Therefore $R_1\cap R_2=\emptyset$.

Given a related set $R$ and a non-empty definable subset $D$ of $R$, we need to show that $D=R$. Let $\Phi(\overline{X})$ be a formula in DNF defining $D$ and let $\Gamma\in\C_\Phi$ define a non-empty set. Since $\Gamma$ defines a subset of $R$, we deduce that $\Gamma$ contains $\Phi_R(\overline{X})$ as a sub-formula. Suppose there is an atomic formula, say $X_i<X_j$, such that $\Phi_R(\overline{X})\wedge(X_i<X_j)$ is still a sub-formula of $\Gamma$. Since $R$ is a related set, precisely one of $X_i=X_j$ or $X_i>X_j$ is an atomic formula appearing in $\Phi_R(\overline{X})$. Hence $\Phi_R(\overline{X})\wedge(X_i<X_j)$ defines the empty set, which contradicts that $\Gamma$ defines a non-empty set. This establishes $D=R$.
\hfill
\end{proof}

By $At^{\emptyset}_n$ we denote the set of related sets in $\Dphi$. Since $\Dphi$ is an atomic boolean algebra, we extend the definition of height to any non-empty definable set in $\Dphi$ by
\begin{equation}
    H(D):=\max\{H(R)\ \big{|}\ R\in At^\emptyset_n,\ R\subseteq D\}.
\end{equation} 

We will now look at formulas definable with a fixed parameter set $\overline{a}$ of size $m$. Henceforth we will assume that the parameter set $\overline{a}$ is in  descending order unless stated otherwise. Let $\Dabar$ denote the set of all the definable sets which could be defined by formulas of $n$ variables and parameter set $\overline{a}$. As one would have guessed $\Dabar$ also forms a \emph{finite} boolean algebra. The following proposition and Definition \ref{d6} are aimed at exploring the structure of this boolean algebra.

Henceforth whenever $D$ is such that $D=\Phi(Q^n;\overline{a})$ for some parameter set $\overline{a}$, we will assume $\Phi$ to be a DNF with positive atomic formulas. 

\begin{proposition}\label{p2}
$\Dabar$ forms a finite boolean algebra with $$\{\Phi_A(\Q^n;\overline{a})\ne\emptyset\ |\ A\in At^{\emptyset}_{n+m}\}$$ as the set of atoms.
\end{proposition}
\begin{proof}
For $A_1,A_2\in At^{\emptyset}_{n+m}$, the disjointness of $\Phi_{A_1}(\Q^n;\overline{a})$ and $\Phi_{A_2}(\Q^n;\overline{a})$ follows from the fact that $\Phi_{A_1}(\Q^{n+m})$ and $\Phi_{A_2}(\Q^{n+m})$ are disjoint. The evaluation map 
\begin{align*}
\mathrm{eval}^{\overline{a}}_{n,m}: \mathcal{D}^{\emptyset}_{n+m} &\longrightarrow\Dabar\\
\Phi_D(\Q^{n+m}) &\longmapsto \Phi_D(\Q^n;\overline{a}).
\end{align*}
is clearly a surjection. If $A$ is an atom below $D$ in the boolean algebra  $\mathcal{D}^{\emptyset}_{n+m}$, then the same holds true of their projections in $\Dabar$ under the evaluation map. For, if $\Phi_A(\Q^n;\overline{a})=\emptyset$, then its trivial. Otherwise if $\Q\models \Phi_A(\overline{b},\overline{a})$ for some $\overline{b}$ then $\Q\models \Phi_D(\overline{b},\overline{a})$ as $\Phi_A(\Q^{n+m})$ is below $\Phi_D(\Q^{n+m})$ which also implies that $\Phi_A(\Q^n;\overline{a})$ is below $\Phi_D(\Q^n;\overline{a})$. Hence the result.
\hfill
\end{proof}

We denote the set of atoms of $\Dabar$ by $\At$. We now define a map which will be important later and explore the structure of the elements of $\At$ by extending our definition of related sets. 
\begin{definition}[Decomposition into related sets]\label{d6}
Define the map
\begin{align*}
\Split:\Dabar &\longrightarrow \mathcal{P}(\At)\\
D &\longmapsto \{A\in \At\ |\ A\cap D\ne\emptyset\}.
\end{align*}
\end{definition}
This map gives us the atomic decomposition of a definable set $D \in\Dabar$. A related set in $\Dabar$ is a non-empty set defined by a formula of the form stated in Equation \eqref{e1} in $n+m$ variables where the final $m$ variables are replaced by the parameters $\overline{a}$. In view of the proposition above, related sets are precisely the atoms in $\Dabar$--the reason we are interested in related sets is because of the form of the formulas used to define them, namely 
\begin{equation}\label{e3}
\Phi_R(X):=\overline{X}_{p_1}<\overline{X}_{p_2}<\ldots< \overline{X}_{p_{k-1}}<\overline{X}_{p_k},
\end{equation}
where $p_i\in\{a_i,e_i\}$ and $\overline{X}_{a_i}$ refers to a tuple, $(X_{i_1},X_{i_2},\cdots,X_{i_{l_i}})$  corresponding to sub-formulas of the form
$$a_i=X_{i_1}=X_{i_2}=X_{i_3}=\cdots=X_{l_i},$$
where $l_i\in\N$ and $X_{i_j}\in \overline{X}$
and $\overline{X}_{e_i}$ refers to a tuple $(X_{i_1},X_{i_2},\cdots,X_{i_{l_i}})$  corresponding to sub-formulas of the form 
$$X_{i_1}=X_{i_2}=X_{i_3}=\cdots=X_{l_i}$$
where $l_i\in\N_{+}$ and $X_{i_j}\in \overline{X}$. For a related set $R$ in $\Dabar$, if the parameter set is clear from the context, by an abuse of notation we will continue to denote the standard formula defining $R$ as in Equation \eqref{e3} by $\Phi_R(\overline{X})$.

The following couple of definitions introduce terminology in order to simplify some proofs later. 
\begin{definition}[Components of a related set]\label{d3}
Define the set of \emph{components} of a related set $R\in At_n^{\overline{a}}$ as 
$$\mathrm{Comp}(R):=\{\overline{X}_{p_i}\ |\ \overline{X}_{p_i}\mbox{ appears in }\Phi_R(\overline{X})\mbox{ as in Equation \eqref{e3}}\}.$$
\end{definition}
Extending Definition \ref{d2} we define the \emph{height} of a related set $R\in At_n^{\overline{a}}$ as
$$H(R):=\mathrm{card}\bigg\{\overline{X}_{p_i}\in \mathrm{Comp}(R)\ \big{|}\ p_i=e_i\bigg\}.$$
One can induce an order relation between the $\overline{X}_{p_i}$s from the order relation on $\overline{X}[1:n]$ appearing in a related set naturally. Further extend the definition of height to $\Dabar$ by defining
\begin{equation}\label{e7}
  H(D):=\max \{H(R)\ |\ R\in \Split(D)\}  
\end{equation}
for any $D\in \Dabar$.
\begin{definition}\label{d4}
Assume $\bar{a}$ is sorted in descending order. Let $A\in\At$. Then for $a_i \in \bar{a}$ define 
$$\#(\overline{X}_{a_i},\overline{X}_{a_{i+1}}):=\mathrm{card}\bigg(\bigg\{\overline{X}_{e_k}\in \mathrm{Comp}(A)\ \big{|}\ \overline{X}_{a_i}<\overline{X}_{e_k}<\overline{X}_{a_{i+1}}\bigg\}\bigg).$$ 
\end{definition}
\begin{definition}[Equivalence up to permutation]\label{d5}
We say that $A_1,A_2\in At_n^{\overline{a}}$ are \emph{equivalent up to permutation}, denoted as $A_1\equiv A_2$, if the following conditions are satisfied:
\begin{enumerate}[leftmargin=*]
\item For  each $1\leq i\leq n-1$,
$\#(\overline{X}_{a_i},\overline{X}_{a_{i+1}})= \#(\overline{Y}_{a_i},\overline{Y}_{a_{i+1}}),$
\item $\mathrm{card}(\{\overline{X}_{e_k}\in \mathrm{Comp}(A_1)\ \big{|}\ \overline{X}_{e_k}<\overline{X}_{a_n}\})= \mathrm{card}(\{\overline{Y}_{e_k}\in \mathrm{Comp}(A_2)\ \big{|}\ \overline{Y}_{e_k}<\overline{Y}_{a_n}\}),$
\item $\mathrm{card}(\{\overline{X}_{e_k}\in \mathrm{Comp}(A_1)\ \big{|}\ \overline{X}_{a_1}<\overline{X}_{e_k}\})= \mathrm{card}(\{\overline{Y}_{e_k}\in \mathrm{Comp}(A_2)\ \big{|}\ \overline{Y}_{a_1}<\overline{Y}_{e_k}\}).$
\end{enumerate}
Where $A_1$ and $A_2$ are represented by disjoint sets of variables $\overline{X}$ and $\overline{Y}$ respectively.
\end{definition}

This definition directly implies whenever $A_1\equiv A_2$, then we have $H(A_1)=H(A_2)$ and $\mathrm{card}(\mathrm{Comp}(A_1))=\mathrm{card} (\mathrm{Comp}(A_2))$.

Now we have finished setting up the necessary background for the next section where we determine which sets are in definable bijection. We finish this section with a result that is the first step in this direction.
\begin{proposition}\label{p3}
If $A_1\equiv A_2$ in $\At$, then there is a definable bijection between $A_1$ and $A_2$. 
\end{proposition}
\begin{proof}
Suppose $A_1,A_2\in\At$  be equivalent up to permutation. We construct a bijection between $A_1$ and $A_2$. Renaming the variables $\overline{X}[n+1:2n]$ by $\overline{Y}[1:n]$ just for the ease of notation. The formula
$$\Psi(\overline{X},\overline{Y}):=\Phi_{A_1}(\overline{X})\wedge \Phi_{A_2}(\overline{Y})\bigwedge\limits_{\substack{\overline{X}_{p}\in \mathrm{Comp}(A_1)}}\overline{X}_{p}=\overline{Y}_{p}$$
is clearly a well-formed formula because $A_1\equiv A_2$. To see that it is the graph of a bijection we evaluate $\Psi(\overline{X},\overline{Y})$ at $\overline{b}\in A_1$ to obtain $\overline{Y}_p=\overline{b}_p$ for each $\overline{Y}_p\in \mathrm{Comp}(A_2)$. This implies each $Y_i=b_{j_i}$ for some $b_{j_i}\in \overline{b}$. Clearly the tuple $(Y_i)_{i=1}^n=(b_{j_i})_{i=1}^n\in A_2$. One can easily extract the bijection $g:A_1\rightarrow A_2$ out of $\Psi(\overline{X},\overline{Y}).$
\hfill
\end{proof}
\section{Computation of $\Deft(\Q)$} \label{s9}
\subsection{Local and global characteristics}
In the previous section we defined a relation $\equiv$ on $\At$ (Definition \ref{d5}). Further Proposition \ref{p3} implies that $\equiv$ is also an equivalence relation. Since our aim is to construct $\Deft(\Q)$ we look at the following quotient $$\Col:=\At/\equiv.$$ 
For $A\in \At$ we denote its equivalence class in $\Col$ by $[A]$. Define, for each $T\in\Col$,  the map 
\begin{align*}
    \LC:\Dabar &\longrightarrow\N\\
    D &\longmapsto \mathrm{card}\big(\{B\in \Split(D)\ |\ B\equiv A\}\big).
\end{align*}
This definition gives rise to the following map
\begin{align*}
    \GC : \Dabar &\longrightarrow \bigoplus_{T\in\Col} \N \\
    D &\longmapsto (\chi_T(D))_{T \in\Col}.
\end{align*}
For $D\in\Dabar$,  $T\in\Col$ we call $\LC(D)$ to be the \emph{local characteristic} of $D$ at $T$ and $\GC(D)$ to be the \emph{global characteristic} of $D$. 
If $D_1, D_2\in\Dabar$  are disjoint and $T\in \Col$, then $\LC(D_1)+\LC(D_2)=\LC(D_1\sqcup D_2)$ and hence 
\begin{equation}\label{e4}
\GC(D_1)+\GC(D_2)=\GC(D_1\sqcup D_2).
\end{equation} 

Proposition \ref{p3} can be extended to $\Dabar$ as follows. Given $D_1, D_2\in\Dabar$, if $\GC (D_1)=\GC(D_2)$, then there is a definable bijection between $D_1$ and $D_2$.

Figure \ref{f1} shows the related sets in $At_2^{ba}$ along with their equivalence classes in $\mathrm{Color}(2,ba)$.

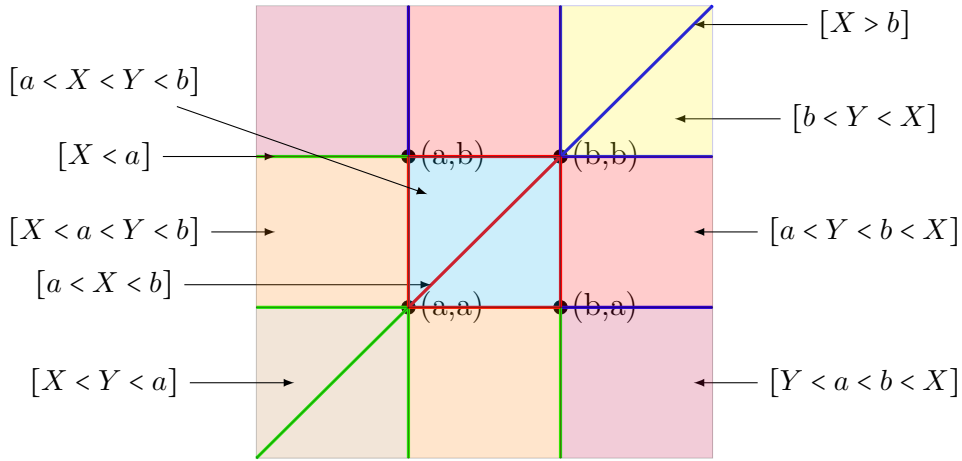
\begin{figure}[!b]
\begin{tikzpicture}
\draw (0,-1)--(6,-1);
\filldraw[black] (4,1) circle (2.5pt) node[anchor=west] {(b,b)};
\draw (0,1)--(6,1);
\filldraw[black] (2,-1) circle (2.5pt) node[anchor=west] {(a,a)};
\draw (0,-3)--(6,3);
\filldraw[black] (2,1) circle (2.5pt) node[anchor=west] {(a,b)};
\draw (2,-3)--(2,3);
\filldraw[black] (4,-1) circle (2.5pt) node[anchor=west] {(b,a)};
\draw (4,-3)--(4,3);

\draw[blue,very thick] (4,1)--(6,3);
\node (blue) at (8,2.75) {\small{$[X>b]$}};
\draw[-latex,black] (blue) to (5.75,2.75);
\draw[blue,very thick] (4,1)--(6,1);
\draw[blue,very thick] (4,1)--(4,3);
\draw[blue,very thick] (4,-1)--(6,-1);
\draw[blue,very thick] (2,1)--(2,3);
\draw[green,very thick] (2,1)--(0,1);
\node (green) at (-2,1) {\small{$[X<a]$}};
\draw[-latex,black] (green) to (.25,1);
\draw[green,very thick] (2,-1)--(0,-1);
\draw[green,very thick] (2,-1)--(0,-3);
\draw[green,very thick] (2,-1)--(2,-3);
\draw[green,very thick] (4,-1)--(4,-3);
\draw[red,very thick] (2,1)--(4,1);
\draw[red,very thick] (4,-1)--(4,1);
\draw[red,very thick] (2,-1)--(4,1);
\node (red) at (-2,-.7) {\small{$[a<X<b]$}};
\draw[-latex,black] (red) to (2.3,-.7);
\draw[red,very thick] (2,-1)--(4,-1);
\draw[red,very thick] (2,-1)--(2,1);
\draw[draw=blue, fill=yellow, opacity=.2] (4,1)--(6,3)--(4,3) -- cycle;
\node (yellow) at (8,1.5) {\small{$[b<Y<X]$}};
\draw[-latex,black] (yellow) to (5.5,1.5); 
\draw[draw=blue, fill=yellow, opacity=.2] (4,1)--(6,1)--(6,3) -- cycle;
\draw[draw=black, fill=red, opacity=.2] (4,-1) rectangle (6,1);
\node (red) at (8,0) {\small{$[a<Y<b<X]$}};
\draw[-latex,black] (red) to (5.75,0);
\draw[draw=black, fill=red, opacity=.2] (4,1) rectangle (2,3);
\node (purple) at (8,-2) {\small{$[Y<a<b<X]$}};
\draw[-latex,black] (purple) to (5.75,-2); 
\draw[draw=black, fill=purple, opacity=.2] (4,-1) rectangle (6,-3);
\draw[draw=black, fill=purple, opacity=.2] (2,1) rectangle (0,3);
\draw[draw=black, fill=orange, opacity=.2] (2,-1) rectangle (4,-3);
\node (orange) at (-2,0) {\small{$[X<a<Y<b]$}};
\draw[-latex,black] (orange) to (.25,0); 
\draw[draw=black, fill=orange, opacity=.2] (2,-1) rectangle (0,1);
\draw[draw=black, fill=brown, opacity=.2] (2,-1)--(0,-3)--(0,-1)--cycle;
\draw[draw=black, fill=brown, opacity=.2] (2,-1)--(2,-3)--(0,-3)--cycle;
\node (brown) at (-2,-2) {\small{$[X<Y<a]$}} ;
\draw[-latex,black] (brown) to (.5,-2); 
\draw[draw=black, fill=cyan, opacity=.2] (2,-1) rectangle (4,1);
\node (cyan) at (-2,2) {\small{$[a<X<Y<b]$}};
\draw[-latex,black] (cyan) to (2.25,0.5); 
\end{tikzpicture}
\caption{Related sets of $D_2^{ba}$ with their equivalence class in $\mathrm{Color}(2,ba)$ (same colored sets belong to same equivalence class).} 
\label{f1} 
\end{figure}

\begin{remark}\label{rmk2}
Given parameter sets $\overline{a}\subseteq\overline{b}$ we have $\Dabar\hookrightarrow\Dbbar$. To see this given a definable subset $D\in\Dabar$, $D$ can also be viewed as an element in $\Dbbar$ with
$$\mathrm{Split}^{\overline{b}}_n(D)=\{R_1\cap R_2\ \big{|}\ R_1\cap R_2\neq\emptyset,\ R_1\in At_n^{\overline{b}},\ R_2\in\Split(D)\}.$$
\end{remark}

The following proposition extends Proposition \ref{p3} in order to identify some condition under which a set in $\Dabar$ is in definable bijection with a set in $\Dbbar$.

\begin{proposition}\label{p4}
Given $D_1\in\Dbbar$ and $D_2\in \Dabar$ if there exists a parameter set $\bar{c}$ such that $\chi_n^{\bar{a}\bar{b} \bar{c}}(D_1)=\chi_n^{\bar{a}\bar{b}\bar{c}}(D_2)$ then there is a definable bijection between $D_1$ and $D_2$.
\end{proposition}
\begin{proof}
The proof is easy and directly follows from Proposition \ref{p3} and Remark \ref{rmk2}.

\hfill
\end{proof}

Let $\Dnq:= \bigcup\limits_{\bar{a}\in\PP_{fin}(Q)} \Dabar$--the directed union is taken over all finite subsets of $Q$. Given a definable set $D\in\Dabar$, we abuse the notation by denoting the equivalence class of $D$ in the directed union $\Dnq$ also by $D$.   Motivated from the above proposition we give the following extension of Definition \ref{d5}.
\begin{definition} \label{defeq} Given $D_1$, $D_2 \in \Dnq$, 
we say $D_1 \equiv D_2$, if there exists a parameter set $\bar{a}$ such that $\GC(D_1)=\GC(D_2)$. \end{definition}

Clearly this is an equivalence relation because of the above proposition. Given $D \in \Dabar$, we will denote its equivalence class in  
$\Dnq/\equiv$ by $[D]$.

The following result is the converse to Proposition \ref{p4} and is the crucial step for identifying the set $\Deft(\Q)$.
\begin{lemma} \label{main-lemma}
If there exists a definable bijection $\Phi$ between definable sets $D_1 \in \Dabar$ and $D_2 \in \Dbbar$, then there exists a parameter set $\overline{c}$ such that  
\begin{align}\chi_{n}^{\overline{a} \overline{b} \overline{c}}(D_1) = \chi_{n}^{\overline{a} \overline{b} \overline{c}}(D_2)).
\end{align}
\end{lemma}

\begin{proof}
WLOG assume that $\Phi$ is given as the following DNF: $\bigvee\limits_{i=1}^m \Phi_{A_i}(\overline{X}[1:2n])$, where, for each $1\leq i\leq m$, $A_i\in \mathrm{Split}^{\overline{a}\overline{b}\overline{c}}_{2n}(Graph(\Phi))$ for some parameter set $\overline{c}$. Consider a particular $A_i$. Note that $A_i$ is also the graph of a definable bijection from some subset of $D_1$ to a subset of $D_2$. Given a $2n$-tuple $\overline{q} \in A_i $, suppose that there exists a $q_u\in \{q_{n+1},\ldots,q_{2n}\}$ such that the following conditions are satisfied:
\begin{enumerate}
    \item $q_j\ne q_{u}\  (\mbox{for each }1\leq j\leq n)$,
    \item $q_{u}\ne d\ \mbox{(for each }d\in\overline{a}\overline{b}\overline{c})$,
\end{enumerate} 
then in $\Phi_{A_i}(\overline{X})$ only following atomic formulas appear:
$$X_{u} > X_j\mbox{ or } X_{u} < X_j \mbox{ (for each }1\leq j\leq n).$$
Condition (2) implies that $X_{u}$  will appear in some $\overline{X}_{e_l}\in\mathrm{Comp}(A_i)$. Further condition (1) implies that $\overline{X}[1:n]\cap\overline{X}_{e_l}=\emptyset$.

Choose $\overline{X}_p, \overline{X}_{p'}\in\mathrm{Comp}(A_i)$ such that $\overline{X}_p<\overline{X}_{e_l}<\overline{X}_{p'}$ and for any $\overline{X}_{p''}\ne\overline{X}_{p},\overline{X}_{p'}\in \mathrm{Comp}(A_i)$, either $\overline{X}_{p''}<\overline{X}_{p}$ or $\overline{X}_{p'}<\overline{X}_{p''}$ holds.
If $\overline{X}_p\cap\overline{X}[1:2n]$ and $\overline{X}_{p'}\cap\overline{X}[1:2n]$ are both non empty and  let $1\leq k,k'\leq 2n$ be such that $X_{k}\in\overline{X}_p$ and $X_{k'}\in\overline{X}_{p'}$. Thus we have $$q_k<q_{u}<q_{k'}.$$ Density of $\Q$ implies that there exists $q''\ne q_{u}$ such that $q_k<q''<q_{k'}$. A similar density argument can be given for the case if either of $\overline{X}_{p'}$ (respectively $\overline{X}_{p}$) is $\overline{X}_d$ for some $d\in\overline{a}\overline{b}\overline{c}$, then we have $q_k<q_u<d$ (respectively $d<q_u<q_{k'}$). 

We now consider the $2n$-tuple $\overline{q'}=(q_1,q_2,\ldots,q_{u-1},q'',q_{u+1},\ldots,q_{2n})$.
Note that $\overline{q'}\in A_i$. Hence for a fixed $n$-tuple $(q_1,\ldots,q_n)\in D_1$ we have $(q_{n+1},\ldots,q_u,\ldots, q_{2n})$,  $(q_{n+1},\ldots,q'',\ldots,q_{2n})\in D_2$. This contradicts the fact that $A_i$ is the graph of a definable bijection. Hence $q_u$ satisfying both the conditions (1) and (2) must not exist. 

Since $A_i$ is the graph of a definable bijection therefore repeat the above procedure for $q_u\in \overline{X}[1:n]$ and taking $j$ in condition (1) as $n+1\leq j \leq 2n$.
Note that failure of one of the conditions (1) or (2) for each $X_j\in \overline{X}[1:2n]$ implies that for each $\overline{X}_{e_v}\in\mathrm{Comp}(A_i)$, $$\overline{X}_{e_v}\cap\overline{X}[1:n]\ne\emptyset \mbox{ and } \overline{X}_{e_v}\cap\overline{X}[n+1:2n]\ne\emptyset.$$ This implies $D_1\equiv D_2$, as elements in $D^{\overline{a}\overline{b}\overline{c}}_n$ (Remark \ref{rmk2} and Definition \ref{d5}). Therefore $ \chi_{n}^{\overline{a} \overline{b} \overline{c}}(D_1) = \chi_{n}^{\overline{a} \overline{b} \overline{c}}(D_2)).$
\hfill
\end{proof} 

\begin{corollary}\label{c1}
If $D_1,D_2\in \D_n^{\overline{a}}$ are such that $H(D_1)\ne H(D_2)$, then there does not exist a definable bijection between them.
\end{corollary}
\begin{proof}
 The proof directly follows from the above lemma.
\hfill
\end{proof}
\subsection{Aggregation with respect to $n$}
Let $D\in\Dnq$ for some $n\geq 1$. Then there is some finite parameter set $\overline{a}$ such that $D=\Phi_D(Q^n;\overline{a})$. Given $n<m\in\N$, define the map $\Delta^Q_{n,m}$ as follows:
\begin{align*}
    \Delta_{n,m}^Q : \D_n^Q &\longrightarrow \D_{m}^Q \\
    D &\longmapsto D'
\end{align*}
where $D'\in\D_{m}^Q$ is defined by the formula 
\begin{equation}\label{e5}
   \Psi(\overline{X}[1:m]):=  \Phi_D(\overline{X}[1:n]) \bigwedge\limits_{i=n+1}^m (X_{i} = X_1). \end{equation}
In the case when $m=n$, define $\Delta_{n,n}^Q$ to be the identity map. 
\begin{remark}\label{rmk3}
Given $D\in \Dnq$, there is an obvious definable bijection between $D$ and $\Delta_{m,n}^Q(D)$.
\end{remark}
The map
$\Delta_{n,m}^{Q}$ naturally induces a map on the quotient $\Dnq/\equiv$ as follows:
\begin{align*}
\overline{\Delta_{n,m}^Q}: \D_n^{Q}/\equiv &\longrightarrow \D_m^{Q}/\equiv\\
[D_1] &\longmapsto [\Delta_{n,m}^Q(D_1)].    
\end{align*}
The above map is well defined. To see this, for some parameter set $\overline{a}$, let $D_1, D_2\in \Dabar$ (Remark \ref{rmk2} allows us to assume this)  be such that $[D_1]=[D_2]$. Hence $\GC(D_1)=\GC(D_2)$.  Equation \eqref{e5} implies that $$\chi_m^{\overline{a}}(\Delta_{n,m}^Q(D_1))=\chi_m^{\overline{a}}(\Delta_{n,m}^Q(D_2)).$$ Therefore $[\Delta_{n,m}^Q(D_1)]=[\Delta_{n,m}^Q(D_2)]$ in $\D_m^Q/\equiv$ (Proposition \ref{p4}).

One readily verifies that $(\D_n^{Q}/\equiv,\overline{\Delta_{n,m}^{Q}})$ forms a directed system. Denote the direct limit of $(\D_n^{Q}/\equiv,\overline{\Delta_{n,m}^{Q}})$ by
\begin{equation}
    \widetilde{\D}(\Q):= \bigcup\limits_{n\in(\N,\leq)}\Dnq/\equiv.
\end{equation}

The following theorem shows that $\Deft(\Q)=\widetilde{\D}(\Q)$ constructed in the above section.
\begin{theorem} \label{th:semiring}
Two definable sets are in definable bijection if and only if they have the same equivalence class in $\widetilde{\D}(\Q)$.
\end{theorem}
\begin{proof}
($\Rightarrow $): In the view of Remark \ref{rmk3} it suffices to assume that $D_1, D_2 \in \Dnq$ for some large enough $n$. If there exists a definable bijection between $D_1$ and $D_2$, then Lemma \ref{main-lemma} implies that there exists a parameter set $\overline{a}$ such that $$\GC(D_1)=\GC(D_2).$$ 
Hence in $\Dnq/\equiv$, we have $[D_1]=[D_2]$ and therefore the equivalence class of $D_1$ is equal to the equivalence class of $D_2$ in $\widetilde{\D}(\Q)$.

($\Leftarrow$): For $n
\leq m$, let $D_1 \in \D_n^{Q}$ and $D_2 \in \D_m^{Q}$. If  the equivalence classes of $D_1$ and $D_2$ are equal in $\widetilde{\D}(\Q)$, then there exists a $k\in\N$ with $n\leq m\leq k$ such that $$ [\Delta_{n,k}^Q(D_1)] = [\Delta_{m,k}^Q(D_2)]\ \ \  \mbox{(by definition of the direct limit).}$$ 
Hence $ \Delta_{n,k}^Q(D_1)$ and $\Delta_{m,k}^Q(D_2)$ are in definable bijection with each other. Finally using Remark \ref{rmk3} we conclude that $D_1$ and $D_2$ are in definable bijection with each other.
\hfill
\end{proof}

In view of this theorem, for $D\in\Dnq$, we will use the notation $[D]$ to denote its class in $\Deft(\Q)$.
\section{Computation of $K_0(\Q)$}\label{s10}
In the previous section we identified the set $\Deft(\Q)$ using Theorem \ref{th:semiring}. In this section we will compute the $K_0(\Q)$ as promised.
As in \S \ref{s7} we endow $\Deft(\Q)$ with the structure of a semiring with unity which turns out to be cancellative.

\begin{theorem}[Cancellativity] \label{cancellative}
The semiring $\Deft(\Q)$ is a cancellative semiring.
\end{theorem}
\begin{proof}
WLOG assume that $A$, $B$ and $C$ are pair-wise disjoint definable subsets, such that 
$$ [A] + [C]= [B] + [C]\ \mbox{in }\Deft(\Q).$$ 
By Theorem \ref{th:semiring} there exists a definable bijection between $A \sqcup C$ and $B \sqcup C$. Assume that $ A \sqcup C$ , $B \sqcup C \in \Dabar$ for some parameter set $\overline{a}$ (Remarks \ref{rmk2} and \ref{rmk3}). Now by Lemma \ref{main-lemma} there exists a parameter set $\overline{c}$ such that 
\begin{equation}\label{e6}
\chi_n^{\overline{a}\overline{c}} (A\sqcup C) = \chi_n^{\overline{a}\overline{c}} (B\sqcup C). 
\end{equation}
Therefore we have 
\begin{align*}
\chi_n^{\overline{a}\overline{c}} (A)+\chi_n^{\overline{a}\overline{c}} (C)&=\chi_n^{\overline{a}\overline{c}} (A\sqcup C)\ \hfill\mbox{(Equation \eqref{e4})}\\
&=\chi_n^{\overline{a}\overline{c}} (B\sqcup C)\ \hfill\mbox{(Equation \eqref{e6})}\\
&=\chi_n^{\overline{a}\overline{c}}
(B)+\chi_n^{\overline{a}\overline{c}}(C)\ \mbox{(Equation \eqref{e4}).}
\end{align*}
Hence we conclude that 
$$\chi_n^{\overline{a}\overline{c}}(A)=\chi_n^{\overline{a}\overline{c}}(B).$$
Therefore $B$ and $A$ are in definable bijection (Proposition \ref{p4}) which completes the proof.
\hfill
\end{proof}

We abuse the notation slightly by using $[D]$ for the equivalence class of $D$ in both $\Deft(\Q)$  and $K_0(\Q)$. The meaning of $[D]$ would be clear from the context. Further for a definable set $D'$ we say that $[D']$ is contained in $[D]$ if there is definable set $D'' \in [D]$ such that $D'\subseteq D''$. Also note that for $D_1$ and $D_2$ such that $[D_1]=[D_2]$ we have $H(D_1)=H(D_2)$ (Corollary \ref{c1} and Theorem \ref{th:semiring}). Hence we naturally extend the definition of height (Equation \eqref{e7}) to $\Deft(\Q)$.

\subsection{Multiplication of two related sets}\label{ss3.1}
It suffices to study the multiplication in the semiring $\Deft(\Q)$, as it is embedded inside $K_0(\Q)$. We will append a formal minimum element, denoted $-\infty$, to $\Q$ in order to simplify the notations in the proof and this has no other effect. 

Given a parameter set $\overline{a}\in Q\sqcup\{-\infty\}$ sorted in  descending order and $\overline{n}=(n_1,n_2,\ldots,n_k)$ be $k$-tuple of non negative integers. Let $^{\overline{a}}R_{\overline{n}}$ be the related set corresponding to the formula  
$$X_1>X_2>\cdots >X_{n_1}>a_1> \cdots> X_{n_1+n_2}>a_2>\cdots> X_{n_1+\ldots+n_k}>a_k$$
and $\big[^{\overline{a}}R_{\overline{n}}\big]\in \Deft(\Q)$ be its equivalence class. When $\overline{a}$ and $\overline{n}$ are tuples of length $1$, then we use the notation $ ^aR_n$ for simplicity. Note that for every parameter set $\overline{a}$, $[ ^{\overline{a}}R_0]$ denotes the equivalence class of singleton in $\Deft(\Q).$
The following proposition states key properties of multiplication of equivalence classes of related sets in $\Deft(\Q)$.
\begin{proposition}\label{p3.1}
The following statements hold in the semiring $\Deft(\Q)$:
\begin{enumerate}
    \item Let $a_1>a_2>a_3\in Q\sqcup\{-\infty\}$ , then $$\big[ ^{(a_1,a_2,a_3)}R_{(0,n,0)}\big]\cdot\big[ ^{(a_1,a_2,a_3)}R_{(0,0,m)}\big]=\big[ ^{(a_1,a_2,a_3)}R_{(0,n,m)}\big],$$ $$\big[ ^{(a_2,a_3)}R_{(n,0)}\big]\cdot\big[ ^{(a_2,a_3)}R_{(0,m)}\big]=\big[ ^{(a_2,a_3)}R_{(n,m)}\big].$$
    \item Let $m\leq n\in\N$ then \begin{equation}\label{e9}
    \big[ ^aR_m\big]\big[ ^aR_n\big]=\sum\limits_{i=0}^m \binom{n+i}{i}\binom{n}{m-i}\big[ ^aR_{n+i}\big].
    \end{equation}
    \item Given $\overline{a}\in Q\sqcup\{-\infty\}$, we have 
    $$[ ^{a_i}R_n]=\sum\limits_{R\in\Split( ^{a_i}R_n)}[R],$$ 
    where each $R$ has form $^{\overline{a}}R_{(m_1,m_2,\ldots,m_i,0,0,\ldots,0)}$ for some non-negetive integers $m_j$ such that $m_1+m_2+\ldots+m_i \leq n$. 
\end{enumerate}
\end{proposition}
\begin{proof}
(1) The proof follows from the definition of `.' operation on $\Deft(\Q)$.

(2) Let $\Phi_{^aR_n}(\overline{X}[1:n])$ and $\Phi_{^aR_m}(\overline{Y}[1:m])$ be the standard formulas corresponding to $^aR_n$ and $^aR_m$ respectively. The idea is to show equality by introducing relations between $\overline{X}[1:n]$ and $\overline{Y}[1:m]$. For each $0\leq i\leq m$, we see that $[ ^aR_{n+i}]$ will be contained in $[ ^aR_m][ ^aR_n]$ and further this will happen precisely when there exist exactly $m-i$ numbers $1\leq l_1,\cdots,l_{m-i}\leq m$ such that, for each $1\leq j\leq m-i$, $Y_{l_1},\cdots,Y_{l_{m-i}}$ satisfy the formula $Y_{l_j}=X_{k_j}$, for some $1\leq k_1,\cdots,k_{m-i}\leq n$. Since any $\overline{q}\in\  ^aR_m\times\ ^aR_n$ will be contained in exactly one of $ ^aR_{n+i}$ with $0\leq i\leq m$, we compute how many times $[ ^aR_{n+i}]$ is contained in $[ ^aR_m\times\ ^aR_n]$. 

We need to take care of the already existing order relations appearing in the formulas corresponding to $^aR_m$ and $^aR_n$. Note that there are $m!$ many related sets in $\D_m^a$ which belong to the equivalence class of $^aR_m$ in $\Deft(\Q)$. Since the equivalence class of the cartesian product of each of these $m!$ sets with $[ ^aR_n]$ contains $[ ^aR_{n+i}]$ equal number of times, therefore we ignore the relative ordering between $\overline{Y}[1:m]$ and divide our answer by $m!$ later to account for this.

Now the choice of $l_1,\cdots,l_{m-i}$ and $k_1,\cdots,k_{m-i}$ can be made in $\binom{m}{m-i}\binom{n}{m-i}(m-i)!$ ways. The variables $\overline{Y}[1:m]\setminus \{Y_{1},\cdots,Y_{l_{m-i}}\}$ can be arranged in $(n+1)(n+2)(n+3) \ldots (n+i)$ possible ways to give a copy of $[ ^aR_{n+i}]$. Therefore the number of times $[ ^aR_{n+i}]$ is contained in $[ ^aR_m\times\ ^aR_n]$ is:
     \begin{align}
        \frac{1}{m!} \Bigg( \binom{m}{m-i} \binom{n}{m-i} (m-i)! (n+1)(n+2)\ldots(n+i) \Bigg).
     \end{align}
     
Upon simplification we get the required coefficient of the summand. Further using the fact that $[ ^aR_{n}]\ne[ ^aR_{m}]$ for $m\ne n$ (Corollary \ref{c1}), we get the required sum.
     
(3) The proof is clear.
\hfill
\end{proof}

\begin{lemma}\label{l3.1}
Given $[ ^{a_i}R_{n_i}]\in\Deft(\Q)$ for $1\leq i\leq k$, and $[D]\in\Deft(\Q)$ such that   
$$[D]=\prod\limits_{\substack{i=0}}^k\big[ ^{a_i}R_{n_i}\big],$$
we obtain that $[ ^{\overline{a}}R_{\overline{n}}]$ is contained in $[D]$ where $\overline{n}=(n_i)_{i=1}^{k}$ and $\overline{a}=(a_i)_{i=1}^{k}$. Further any other $^{\overline{a}}R_{\overline{n'}}\in \mathrm{Split}^{\overline{a}}_{n_1+\cdots+n_k}(D)$ such that $H( ^{\overline{a}}R_{\overline{n'}})=\sum\limits_{i=1}^k n_i$ has the property that $n'_j<n_j$ for some $1\leq j\leq k$. Also $\#(\overline{X}_{a_{k-1}}, \overline{X}_{a_k})$ has the maximum value for $^{\overline{a}}R_{\overline{n}}$ among all $^{\overline{a}}R_{\overline{n'}}\in \mathrm{Split}^{\overline{a}}_{n_1+\cdots+n_k}(D)$.                        
\end{lemma}
\begin{proof}
The proof of the lemma is easy. From part (3) of the above proposition we know that each $\big[  ^{a_i}R_{n_i}\big]$ contains $^{\overline{a}}R_{0,0,\ldots,n_i,0,\ldots,0}$ and for any other $R\in\Split( ^{a_i}R_{n_i})$ with height $R=n_i$, we have $\#(\overline{X}_{a_{i-1}},\overline{X}_{a_i})<n_i$. Finally using parts (1), (2) of the above proposition we have the lemma.
\hfill 
\end{proof}

Consider the polynomial ring over integers generated by $\N_+\times(\mathbb Q\sqcup\{\-\infty\})$, where the generator corresponding to the pair $(n,a)$ is denoted $ ^aX_n$. Further consider its ideal given by
\begin{align} \label{ideal}
I:=\bigg\la  ({^aX_k} {^{a}X_l}) - \sum\limits _{i=0}^{l} \binom{k+i}{i}\binom{k}{l-i} {^aX_{k+i}} \ |\ 0\leq l\leq k,\ \forall k \in \N_{+} ; a \in Q \sqcup \{-\infty\} \bigg\ra.
\end{align}

\begin{theorem}\label{iso}
With all the above notations we have $K_0(\Q)\simeq \mathcal{O}$, where 
$$
\mathcal{O} := \mathbb{Z}\big[{^aX_n\ |\ n \in \N_{+}, a \in Q \sqcup \{-\infty\}}\big]/I.$$
\end{theorem}

\begin{proof}
  
Define the association map 
\begin{align*}
    \zeta : \{^aX_n\ |\ a\in Q, n\in\N_{+}\}&\longrightarrow K_0(\Q)\\
    ^aX_n &\longmapsto [ ^aR_n]
\end{align*}
Extend the association map $\zeta$ naturally to a ring homomorphism from $\mathcal{O}$ to $K_0(\Q)$ and denote it also by $\zeta$. We will show that $\zeta$ is an isomorphism. $\zeta$ is well defined from part (2) of Proposition \ref{p3.1}. Surjectivity  of $\zeta$ is easy to show and follows from part (1) and (3) of Proposition \ref{p3.1}.

Now we will show that $ \zeta$ is an injection. Let $f\in\mathbb{Z}\big[{^aX_n\ |\ n \in \N_{+} ; a \in Q \sqcup \{-\infty\}}\big]$, let $\overline{f}$ denote the corresponding element in $\mathcal{O}$. For some $l\in\N_{+}$, let $g=t\prod\limits_{i=1}^l {(^{a_i}X_{n_i})^{m_i}}$ be a monomial of $f$, where $n_i,m_i\in\N$ for each $1\leq i\leq l$ and $t\ne0\in \Z$. Define the height of $g$, denoted $H(g)$, by
$$H(g):=H\Big(\zeta\big(\overline{\prod\limits_{i=1}^l {(^{a_i}X_{n_i})^{m_i}}}\big)\Big)=\sum\limits_{i=1}^k n_im_i.$$ 
Extend the definition of height to $f$ as $$H(f):=\max\{ H(g)\ |\ g\mbox{ is a monomial of }f\}.$$ 
Now each $(^aX_n)^m\equiv t$ $^aX_{nm}+ \mbox{ lower height terms }\mbox{(mod I)} $ for some $t\in\N$ obtained from part (2) of Proposition \ref{p3.1}. Hence every monomial $g$ of $f$ can be written as 
\begin{equation}\label{e8}
g\equiv t'\prod\limits_{i=1}^l {^{a_i}X_{n_im_i}}+\mbox{ lower height terms (mod I)}
\end{equation}
for some $t'\ne0\in \Z$. Note that if $f\equiv f' \mbox{ (mod I)}$ then $H(f)=H(f').$ 
%Therefore we define $H(\overline{f})=H(f)$. 

Choose $\overline{f}\in \mbox{Ker}(\zeta)$ such that $H(f)$ is minimum. Let $f'\equiv f \mbox{ (mod I)}$ be such that every monomial $g$ of $f'$ with $H(g)=H(f)$ has the form $$c'\prod\limits_{i=1}^l {^{a_i}X_{n_i}}$$
for some $l\in \N_{+}$.
Equation \eqref{e8} implies that such an $f'$ exists. 
Let 
$$\overline{b}:=\{q\ |\ ^qX_m \mbox{ appears in some monomial of }f\mbox{ for some }m\}$$ 
and $k=\mathrm{card}(\overline{b})$.
Assume $k>1$ otherwise the proof is trivial (part (2) of Proposition \ref{p3.1}). Let $$f'=f_1-f_2\textit{ ,}$$ where $f_1,f_2$ have all positive coefficients. We have $\zeta(\overline{f'})=0$, hence $\zeta(\overline{f_1})=\zeta(\overline{f_2})$ in $\Deft(\Q)$. WLOG assume that $\overline{b}$ is ordered in descending order. Let $$g_{max}=\prod\limits_{i=1}^k {^{b_i}X_{n_i}}$$ with $H(g_{max})=H(f)$ be the monomial of $f$ such that $n_k$ is maximum among all monomials of $f$. If there exist two or more such monomials then among them choose the one in which $n_{k-1}$ is maximum and so on. 
Note that this process will render us a unique $g_{max}$. WLOG assume that $g_{max}$ is in $f_1$. 
We claim that $\big[^{\overline{b}}R_{\overline{n}}\big]$ where $\overline{b}=(b_1,\ldots,b_k)$ and $\overline{n}=(n_1,\ldots,n_k)$ is contained in $\zeta(\overline{g}_{max}) \in \Deft(\Q)$ and $\big[^{\overline{b}}R_{\overline{n}}\big]$ is not contained in $\zeta(\overline{g})$ for any other monomial $g$ of $f$. To see this suppose there is $g\ne g_{max}$ and $\zeta(\overline{g})$ contains $\big[^{\overline{b}}R_{\overline{n}}\big]$. WLOG assume 
$$g=\prod\limits_{j=1}^l \big(^{b_j}X_{n_j'}\big)\prod\limits_{i=l+1}^{k}\big(^{b_i}X_{n_i}\big)$$
for some $l\leq k$. By construction we see that $n_l'<n_l$. Lemma \ref{l3.1} implies that $\zeta\Big(\overline{\prod\limits_{j=1}^l\  ^{b_j}X_{n_j'}}\Big)$ contains $\big[^{(b_1,\ldots,b_l)}R_{(n_1',\ldots,n_l')}\big]$ and among all $[R]$ contained in $\zeta\Big(\overline{\prod\limits_{j=1}^l({^{b_j}}X_{n_j'}})\Big)$, $\#(\overline{X}_{b_{l-1}},\overline{X}_{b_l})$ attains maximum value for $^{(b_1,\ldots,b_l)}R_{(n_1',\ldots,n_l')}$. Since for $i>l$ we have $b_i<b_l$, by part (3) of Proposition \ref{p3.1}, $\zeta(\overline{g})$ does not contain $\big[^{\overline{b}}R_{\overline{n}}\big]$ in $\Deft(\Q)$.  Hence the coefficient of $\overline{g}_{max}=0$ which is a contradiction, therefore $ \overline{f}=0$. This establishes that $\zeta$ is an isomorphism between $K_0(\Q)$ and $\mathcal{O}$. 
\hfill
\end{proof}

\section{DLO with end points}\label{s12}
Let $\Q$ denote a DLO without end points, and let $_m\Q,\ \Q_M,\ _m\Q_M$ denote DLOs obtained by appending to $\Q$ only the minimum end point $m$, only the maximum end point $M$ and both end points $m$ and $M$ respectively. We will restrict our discussion to $_m\Q$ as other cases can be dealt similarly.
Let $\T'$ represent the theory of DLO with only the minimum point. This theory does not admit complete elimination of quantifiers, but we view a definable subset of any of its models as a definable subset of a DLO without end points as explained below.

The structure $\mathrm{Ext}_{\Q}$ obtained by appending a copy of $\Q$, say $\Q'$, below the minimum of $_m\Q$ is a model of theory $\T$. Consider a definable subset $D$ of $_m\Q$ and let $\Phi_D$ be the formula corresponding to it. Observe that $D$ can be seen as a definable subset in $\mathrm{Ext}_{\Q}$. Let $\overline{a}\in Q\sqcup\{m\}$ be the set of parameters appearing in $\Phi_D$ and $n$ be the number of variables appearing in $\Phi_D$. Then the collection of definable subsets of $\mathrm{Ext}_\Q$ with parameters in the set $Q\sqcup \{m\}$ is same as the collection of definable subsets of $_m\Q$. Proposition \ref{p2} implies existence of $R_1, R_2,\cdots,R_k\in \At(\mathrm{Ext}_\Q)$ such that $D=\bigsqcup\limits_{i=1}^k R_i$, where each $R_i$ is a subset of $_m\Q$. 
\begin{lemma}\label{lemma-endpt}
For $D_1,D_2\in \overline{\mathrm{Def}}( _m\Q)$, we have $[D_1]=[D_2]$ in $\Deft( _m\Q)$ iff $[D_1]=[D_2]$ in $\Deft(\mathrm{Ext}_{\Q})$.
\end{lemma}
\begin{proof}
($\Rightarrow$): The proof is clear. 

($\Leftarrow$): If $[D_1]=[D_2]$ in $\Deft(\mathrm{Ext}_{\Q})$ then there exists a parameter set $\overline{c}\in Q'\sqcup\{m\}\sqcup Q$ such that $$\chi_n^{\overline{c}}(D_1)=\chi_n^{\overline{c}}(D_2).$$ Let
$\overline{c'}:=\overline{c}\cap (Q\sqcup\{m\})$. If $\overline{c'}=\overline{c}$ there is nothing to be proven. Suppose not, then let $l=\mathrm{card}(\overline{c})$ and $k$ be the maximum index such that $c_k\geq m$. As $D_1$ and $D_2$ are definable subsets of $_m\Q$ we have for each related set $R\in \mathrm{Split}_n^{\overline{c}}(D_1)$, $$\#(\overline{X}_{c_{k}},\overline{X}_{c_{k+1}})=\#(\overline{X}_{c_{k+1}},\overline{X}_{c_{k+2}})=\cdots=\#(\overline{X}_{c_{l-1}},\overline{X}_{c_l})=0$$ which implies that each $R\in \mathrm{Split}_n^{\overline{c}}(D_1)$ (similarly for $D_2$) can be viewed as a related set in $D_n^{\overline{c'}}$. Hence a definable bijection between $D_1$ and $D_2$ as subsets of $\mathrm{Ext}_{\Q}$ is also a definable bijection between $D_1$ and $D_2$ as subsets of $_m\Q$.
\hfill
\end{proof}

The above lemma implies that $\Deft(_m\Q)\hookrightarrow\Deft(\mathrm{Ext}_{\Q})$ as $\Deft(\mathrm{Ext}_{\Q})$  is generated by elements of the form $[^aR_n]$ for $n\in\N$ and $a\in Q'\sqcup\{m\}\sqcup Q$. Hence every element in $\Deft(_m\Q)$ is generated by elements of the form $[^aR_n]$ for $n\in\N$ and $a\in Q\sqcup\{m\}$. Cancellativity (Lemma \ref{cancellative}) will ensure that $K_0( _m\Q)\hookrightarrow K_0(\mathrm{Ext}_{\Q})$. In view of the above discussion we have the following theorem.
\begin{theorem}\label{wep}
$K_0(_m\Q)\simeq K_0(\Q)$.
\end{theorem}
\begin{proof}
Note that it suffices to show that $\Deft(_m\Q)\simeq\Deft(\Q)$. We have the following semiring embeddings, $\Deft(\Q)\hookrightarrow \Deft(\mathrm{Ext}_{\Q})$ and $\Deft(_m\Q)\hookrightarrow\Deft(\mathrm{Ext}_{\Q})$. It only remains to find a representative of the  equivalence class $[{^mR_n}]\in \Deft(_m\Q)$ in $\Deft(\Q)$. For this observe that that $[^{-\infty}R_n]\in\Deft(\Q)$ corresponds to the equivalence class of the related set corresponding to the formula $X_1<X_2<\cdots<X_n$, and this class gets mapped to $[{^mR_n}]$   under the embedding $\Deft(\Q)\hookrightarrow\Deft(\mathrm{Ext}_{\Q})$. This implies that $\Deft(\Q)$ and $\Deft(_m\Q)$ go to the same subset of $\Deft(\mathrm{Ext}_{\Q})$ under their natural embedding as a subsemiring. \hfill
\end{proof}
\begin{example}
Let $\mathbb{Q}_+$ be the set of positive rational numbers and $   \mathbb{Q}_{\geq 0}$ be the set of non-negative rational numbers. In the view of Lemma \ref{lemma-endpt} and Theorem \ref{wep} we have $K_0((\mathbb{Q}_+,<))\simeq K_0((\mathbb{Q}_{\geq 0},<))\hookrightarrow K_0((\mathbb{Q},<))$. Given any order isomorphism between $\mathbb Q$ and $\mathbb Q_+$, we get an isomorphism between $K_0(\mathbb Q)$ and $K_0(\mathbb Q_+)$; thus we witness the failure of PHP for $K_0(\mathbb Q)$ as a structure of the signature of rings!
\end{example}
\section{Combinatorial Properties of $K_0(\Q)$}\label{s11}
In this section we discuss some interesting combinatorial properties of $K_0(\Q)$ for a DLO without end points $\Q$. Note that the properties will be true in general for any DLO $\Q$ due to Theorem \ref{wep}.

First we note that a DLO satisfies PHP.
\begin{theorem}\label{PHP}
The ring $K_0(\Q)$ is partially ordered or, equivalently, $\Q$ satisfies PHP.
\end{theorem}
\begin{proof}
Lemma \ref{main-lemma} implies that there cannot be a definable bijection between a definable set and its proper definable subset. Hence in view of Theorem \ref{ts8.2} we have the result.
\hfill
\end{proof}

The classes of related sets in $K_0(\Q)$ satisfy some nice convolution-type relations.
\begin{theorem}\label{convolution}
For $a,b\in Q\sqcup\{-\infty\}$ and $a<b$, let $$^nf(b,a):=\zeta^{-1}(\big[^{(b,a)}R_{(0,n)}\big]),$$ then 
\begin{enumerate}
    \item for any $c\in Q\sqcup\{-\infty\}$ such that $a<c<b$ we have 
$$^nf(b,a)=\sum\limits_{\substack{i=0}}^n {^if(b,c)}\ { ^{n-i}f(c,a)}+\sum\limits_{\substack{i=0}}^{n-1} {^if(b,c)}\  { ^{n-1-i}f(c,a)},$$
\item $(n!)( { ^nf(b,a)})=\prod\limits_{i=0}^{n-1} ( ^1f(b,a)-i).$
\end{enumerate}
\end{theorem}
\begin{proof}
\begin{enumerate}[leftmargin=*]
    \item Recall that $\zeta$ is an isomorphism. Hence $^nf(b,a)$ satisfies same relations as $[^{(b,a)}R_{0,n}]$. Therefore by part (1) and (3) of Proposition \ref{p3.1} we have the result.
    \item Using part (2) of Proposition \ref{p3.1} putting $m=1$ and $n=n-1$ we have $n[ ^aR_{n}]+(n-1)[ ^aR_{n-1}]=[ ^aR_{n-1}][ ^aR_1]$. Hence we have
    $$n ^nf(b,a)+(n-1) ^{n-1}f(b,a)=\  ^{n-1}f(b,a)\  ^1f(b,a).$$
    Repeated use of the above equation gives us the required result.
\end{enumerate}
\hfill
\end{proof}

The ideal of the polynomial ring defined by Equation \eqref{ideal} has a nice property that each of its element has a positive integer multiple in a  smaller ideal.
\begin{theorem} Let $I$ be the ideal of the polynomial ring as in Equation \eqref{ideal}. Let
\begin{equation}\label{small_ideal}
  I':=\bigg\langle (k!)  (^aX_k) - \prod\limits_{i=0}^{k-1} (^aX_1 - i) \ | \  k \in \N_{+}; a \in Q \sqcup \{-\infty\} \bigg\rangle  
\end{equation}
be an ideal contained in $I$. Suppose for each $k\in\N_{+}$ we denote the formal variable $^aX_k$ by $X_k$. Then for a fixed $a\in Q\sqcup\{-\infty\}$ we have
\begin{align}
   \bigg({X_k} \ \prod\limits _{i=0}^{l-1} (X_1 - i) \bigg) -\big( l! \big) \bigg(\sum\limits _{i=0}^{l} \binom{k+i}{i}\binom{k}{l-i} {X_{k+i}} \bigg)\equiv 0 \ (\mbox{mod }I')\quad\forall l\leq k \in \N_{+}.
\end{align}
\end{theorem}
\begin{proof}
By part (2) of Theorem \ref{convolution} we have:
\begin{equation}\label{eq-special}
    k! X_k \equiv X_1(X_1 -1)(X_1-2)\ldots(X_1 - (k-1))\  (\mbox{mod }I'),
\end{equation}
and by proof of part (2) of Theorem
\ref{convolution}
we have 
\begin{equation}\label{e10}
     \prod\limits_{i=1}^l(k+i) X_{k+l} = X_k (\prod\limits_{j=0}^{l-1}(X-k-j)\ (\mbox{mod }I').
\end{equation}
Equation \eqref{e9} in terms of formal variables $X_k$ and $X_l$ can be equivalently written as 
\begin{equation}\label{e11}
l! X_k X_l \equiv l! X_k \binom{k}{l} + \sum\limits_{i=1}^l \big(\ ^lC_i\ ^kP_{l-i} X_{k+i}(\prod\limits_{s=1}^i(k+s))\big) \ (\mbox{mod }I).
\end{equation}
Using Equation \eqref{e10} and Equation \eqref{e11} our work reduces to showing 
\begin{equation}\label{e12}
l! (X_l) \equiv (^kP_l) + \sum\limits_{i=1}^l \big(\ ^lC_i\ ^kP_{l-i}\prod\limits_{j=0}^{i-1}(X_1 - k- j)\big) \ (\mbox{mod }I').
\end{equation}
Let
\begin{equation*}
    P_1(X) := \ ^kP_l + \sum\limits_{i=1}^l \big(\ ^lC_i\ ^kP_{l-i}\ 
\prod\limits_{j=0}^{i-1}(X - k- j)\big)
\end{equation*}
and 
\begin{equation*}
    P_2(X) : = X(X -1)(X-2)\ldots(X- l+1).
\end{equation*}
Note that it suffices to show  $P_1(X)=P_2(X) $ in $\Z[X]$ (Equation \eqref{e10}).
We evaluate $P_1(X)$ at $X = k+ t$, for each $0\leq t\leq l$ and show that $P_2(k+t) - P_1(k+t) = 0$. Since we are working in the integral domain ($\mathbb{Z}[X]$), polynomial of degree $l$ can have at most $l$ zeros, which will force $P_1 - P_2 $ to be identically zero. 

The proof is via induction on $t$. For $t=0$, checking is straightforward as the summation term vanishes.
For $t = r$, assume that $P_1(k+r) = P_2(k+r)$. Note that $P_2(k+r)=\ ^{k+r}P_l$. Now, for $t=r+1$, 
\begin{align*}
P_1(k+r+1) &=\ ^kP_l+ \sum\limits_{i=1}^{l}\big( \ ^lC_i\ ^kP_{l-i}\prod\limits_{j=0}^{i-1 }(r+1 - j)\big)\\ 
&=\ ^kP_l+ \sum\limits_{i=1}^{r+1} \big(\ ^lC_i\ ^kP_{l-i}\prod\limits_{j=0}^{i-1 }(r+1 - j)\big)\ \mbox{(for }i>r+1\mbox{ product term vanishes)}\\
&=\  ^kP_l + \sum\limits_{i=1}^r \big( \ ^lC_i\ ^kP_{l-i}\prod\limits_{j=0}^{i-1 }(r-j)\big) +   \sum\limits_{i=1}^r \big(\ ^lC_i\ ^kP_{l-i} \prod\limits_{j=0}^{i-1}(r+1 - j) - \prod\limits_{j=0}^{i-1}(r - j)  \big)\\ 
    &+\ ^lC_{r+1}\ ^kP_{l-r-1} \prod\limits_{j=0}^r(r+1-j).
\end{align*}
Note that, if $i>1$, then
\begin{equation}\label{e13}
\prod\limits_{j=0}^{i-1}(r+1 - j) - \prod\limits_{j=0}^{i-1}(r - j) =  i\prod\limits_{j=0}^{i-2}(r-j),
\end{equation}
  and for $i=1$, $\prod\limits_{j=0}^{i-1}(r+1 - j) - \prod\limits_{j=0}^{i-1}(r - j) = 1$. Using induction hypothesis we have   
  $$^kP_l + \sum\limits_{i=1}^r\big( \ ^lC_i\ ^kP_{l-i}\prod\limits_{j=0}^{i-1 }(r-j)\big)=\ ^{k+r}P_l.$$ 
  Using Equation \eqref{e13}, we have
  \begin{align*}
   P_1(k+r+1)&=\ ^{k+r}P_l + l\ ^kP_{l-1} 
   + l\Big(\sum\limits_{i=2}^{r} \big(\ ^{l-1}C_{i-1}\ ^kP_{l-i} \prod\limits_{j=0}^{i-2}(r-j)\big)\Big)\\  &+ \ ^lC_{r+1}\ ^kP_{l-r-1}\prod\limits_{j=0}^r(r+1-j).
 \end{align*}
We foucs on the term $l ^kP_{l-1} 
   + l\Big(\sum\limits_{i=2}^{r} \big((^{l-1}C_{i-1}\ ^kP_{l-i}\prod\limits_{j=0}^{i-2}(r-j)\big)\Big) + \ ^lC_{r+1} \ ^kP_{l-r-1}$ in the above equation. Re-indexing gives us 
\begin{equation}\label{e14}
    l \big(^kP_{l-1} \ + \ \sum\limits_{i=1}^{r-1} \ ^{l-1}C_{i}\ ^kP_{l-i-1} \prod\limits_{j=0}^{i-1}(r-j)\big)  + \ ^lC_{r+1}\ ^kP_{l-r-1}\prod\limits_{j=0}^r(r+1-j). 
    \end{equation} 
Note that
\begin{align*}
    ^lC_{r+1}\ ^kP_{l-r-1} \prod\limits_{j=0}^r(r+1-j)&=\ ^{l-1}C_{r}
\bigg(\frac{l}{r+1}\bigg) \ ^kP_{l-r-1} \prod\limits_{j=0}^r(r+1-j)\\
&=l\ ^{l-1}C_{r} \ ^kP_{l-r-1} \prod\limits_{j=1}^r(r+1-j)\\
&= l\ ^{l-1}C_{r} \ ^kP_{l-r-1} \prod\limits_{j=0}^{r-1}(r-j).
\end{align*}
Combining this with Equation \eqref{e14}
we have 
\begin{align*}
    P_1(k+r+1)&=\ ^{k+r}P_l+ l\ ^kP_{l-1}+ l\sum\limits_{i=1}^{r}\ ^{l-1}C_{i}\ ^kP_{l-i-1}\prod\limits_{j=0}^{i-1}(r-j))\\
    &=\ ^{k+r}P_l + l\ ^{k+r}P_{l-1}\\
    &=\ ^{k+r+1}P_l\\
    &= P_2(k+r+1).
    \end{align*}
This establishes the theorem.
\hfill
\end{proof}

\begin{corollary}
Let $I, I'$ be the ideals defined by Equations \eqref{ideal} and \eqref{small_ideal} respectively.
Given an element $ \alpha\in I$ there exists an integer $n_{\alpha}$ such that $(n_{\alpha}!) \alpha \in I'$.
\end{corollary}
\begin{proof}
Directly follows from the above theorem.
\hfill
\end{proof}

\begin{rmk}\label{CCrmk}
It can be easily observed from Lemma \ref{main-lemma} that the combinatorial principles CC1 and CC2 (Definition \ref{CC}) fail in a DLO without end points.
\end{rmk}

\addcontentsline{toc}{section}{References}
\bibliography{bibtex.bib}
\bibliographystyle{plain}

\end{document}